\documentclass[11pt]{article}
\usepackage{amsthm,mathrsfs,amssymb}
\usepackage{baskervald}
\usepackage{amsmath,verbatim}
\usepackage{hyperref}
\usepackage[numbers]{natbib}
\usepackage[active]{srcltx}
\usepackage[matrix,arrow,curve]{xy}
\usepackage{bbm}
\usepackage{verbatim}
\usepackage{setspace}
\usepackage{parskip}
\usepackage[margin=2.5cm]{geometry}
\usepackage{baskervald}

\newtheoremstyle{mytheoremstyle} 
{\topsep}                    
{\topsep}                    
{\itshape}                   
{1.5em}                           
{\itshape\bf}                   
{.---}                          
{.50em}                       
{}  
\theoremstyle{mytheoremstyle}

\newtheorem{para}{}[subsection]
\newtheorem{thm}[para]{Theorem}
\newtheorem{prop}[para]{Proposition}
\newtheorem{lemma}[para]{Lemma}
\newtheorem{cor}[para]{Corollary}

\theoremstyle{definition}
\newtheorem{defn}{Definition}[section]

\numberwithin{equation}{section}
	\usepackage{verbatim}

\newcommand{\scrD}{\mathscr{D}}

\newcommand{\scrR}{\mathscr{R}}

\newcommand{\calD}{\mathcal{D}}

\newcommand{\calO}{\mathcal{O}}

\newcommand{\frakf}{\mathfrak{f}}
\newcommand{\frakg}{\mathfrak{g}}
\newcommand{\frakh}{\mathfrak{h}}

\newcommand{\frakk}{\mathfrak{k}}
\newcommand{\frakl}{\mathfrak{l}}

\newcommand{\frakp}{\mathfrak{p}}

\newcommand{\fraks}{\mathfrak{s}}
\newcommand{\frakt}{\mathfrak{t}}
\newcommand{\fraku}{\mathfrak{u}}

\newcommand{\frakS}{\mathfrak{S}}

\newcommand{\bfw}{{\bf w}}

\newcommand{\CC}{\mathbb{C}}

\newcommand{\HH}{\mathbb{H}}

\newcommand{\QQ}{\mathbb{Q}}
\newcommand{\RR}{\mathbb{R}}

\newcommand{\ZZ}{\mathbb{Z}}

\renewcommand{\ker}{\operatorname{ker}}

\newcommand{\so}{\Longrightarrow}
\newcommand{\isomto}{\xrightarrow{\raisebox{-5pt}{$\sim$}}}

\newcommand{\rar}{\rightarrow}

\newcommand{\lra}{\longrightarrow}

\newcommand{\hra}{\hookrightarrow}

\newcommand{\wt}{\widetilde}

\newcommand{\ol}{\overline}

\newcommand{\sgn}{\mathrm{sgn}}

\newcommand{\End}{\mathrm{End}}
\newcommand{\Hom}{\mathrm{Hom}}

\newcommand{\Lie}{\mathrm{Lie}}

\newcommand{\ad}{\mathrm{ad}}

\newcommand{\Tr}{\operatorname{Tr}}

\newcommand{\id}{\mathrm{id}}

\newcommand{\Sp}{\operatorname{Sp}}

\newcommand{\Ind}{\mathrm{Ind}}
\newcommand{\Rep}{\mathrm{Rep}}

\newcommand{\rank}{\mathrm{rank}}

\newcommand{\Sym}{\mathrm{Sym}}
\newcommand{\op}{\mathrm{op}}

\renewcommand{\ker}{\operatorname{ker}}

\newcommand{\GL}{{\mathrm{GL}}}

\newcommand{\SU}{{\mathrm{SU}}}

\newcommand{\cleq}{\preccurlyeq}

\newcommand{\Alt}{\mathrm{Alt}}

\newcommand{\mtwo}[1]{\left(\begin{array}{cc} #1 \end{array}\right)}
\newcommand{\mone}[1]{\left(\begin{array}{c} #1 \end{array}\right)}

\renewcommand{\mod}{\text{mod }}

\newcommand{\rN}{\mathrm{N}}

\newlength{\dhatheight}

\newcommand{\qi}{\hat{\iota}}
\newcommand{\rT}{\mathrm{T}}
\newcommand{\qj}{\hat{\jmath}}
\newcommand{\qk}{\hat{\scriptstyle k}}

\newcommand{\su}{\fraks \fraku}
\newcommand{\slf}{\fraks \frakl} 

\newcommand{\pard}[2]{{\frac{\partial{#1}}{\partial{#2}}}}

\title{Fueter-Regular Discete Series for $\Sp(1,1)$}
\author{Z. Amir-Khosravi}
\begin{document}
\maketitle
\begin{abstract}We show that the quaternionic discrete series on $G=\Sp(1,1)$ with minimal $K$-type of dimension $n+1$ can be realized inside the space of Fueter-regular functions on the quaternionic ball $B\subset \HH$, with values in $\HH^n$. We then consider the corresponding $\HH^n$-valued Fueter-regular automorphic forms on $G$. For a fixed level $\Gamma$, we construct a non-trivial map from the space of pairs of such automorphic forms, to closed $M_n(\HH)$-valued differential $3$-forms on $B$, which transform under $\Gamma$ according to a cocycle condition.  
\end{abstract}

\setcounter{tocdepth}{2}

\tableofcontents

\section{Introduction}

Real semisimple Lie groups that have compact Cartan subgroups admit representations in the discrete series, by the work of Harish-Chandra \cite{HC65,HC66}. If the symmetric space of the group is hermitian, among the discrete series are the  holomorphic ones, which can be realized inside spaces of holomorphic functions. This fact plays a fundamental role in linking the theory of holomorphic automorphic forms to complex and algebraic geometry.

W. Schmid in his thesis \cite{Sch89} showed how to construct discrete series in general, as long as the corresponding parameter is sufficiently nonsingular, as the kernel of a certain Dirac operator $\scrD$ acting on vector-valued functions on the group. When the symmetric space is hermitian, $\scrD$ may coincide with the Cauchy-Riemman operator $\ol{\partial}$, in which case one recovers the holomorphic discrete series.

If the maximal compact subgroup has an $\SU(2)$ factor, the symmetric space is a quaternionic-K\"ahler manifold, and the group admits \textit{quaternionic discrete series}. These were classified by B. Gross and N. Wallach \cite{GW96}, who showed they can be realized inside the first cohomology group of a invertible sheaf on the associated twistor space. 

More recently, H. Liu and G. Zhang \cite{LiuZhang} showed that the quaternionic discrete series on $\Sp(1,n)$ can be constructed inside the space of functions on the quaternionic $n$-ball, with values in complex vector spaces, that are annihilated by a relatively simple Dirac operator.

In the 1930s, R. Fueter and his students had studied a quaternionic analogue of holomorphic functions \cite{Fuet36}. A function of a quaternionic variable $q=t+x\qi + y\qj + z\qk$ is called \textit{Fueter regular}, if it is annihilated by the (left-sided) Fueter operator
$$ \ol{\partial}_l = \pard{}{t} + \qi \pard{}{x} + \qj \pard{}{y} + \qk \pard{}{z}.$$
Fueter regular functions have similar properties as holomorphic functions. They are harmonic, coincide with certain convergent series, and satisfy analogues of theorems of Cauchy and Liouville, among others, from standard complex analysis. The Dirac operator of Liu-Zhang coincides with the complex form of the Fueter operator in the \textit{limit of discrete series} case.

In this paper we first show that for the group $\Sp(1,1)$, whose symmetric space is the quaternionic ball $B$, \textit{all} the quaternionic discrete series can be realized inside spaces of $\HH^n$-valued Fueter-regular functions on $B$. We do this by applying a transform to the space of functions on which the Liu-Zhang operator acts.

We then consider Fueter-regular automorphic forms on a rational form of $\Sp(1,1)$, corresponding to vectors in the quaternionic discrete series fixed by a discrete subgroup $\Gamma$. To a pair of such automorphic forms we associate a closed differential $3$-form on the quaternionic ball, with values in a quaternionic vector space, that transform under $\Gamma$ by an automorphy factor. 

Now we describe the results in further detail. We fix an isomorphism between $\SU(2)$ and the norm-one quaternions $\HH^1$. Let $R_n: \GL_2(\CC) \rar \GL(V_n)$ denote the representation of $\HH^\times$ on homogeneous polynomials of degree $n$ in $\CC[X,Y]$, by considering $\HH^\times \subset \GL_2(\CC)$. We consider $\HH^{n}$ as a space of row vectors, and define the action of $G$ on $C^\infty(B,\HH^{n})$ by
\begin{align}\label{GL2rep}(\gamma\cdot f) =  \frac{(cq+d)^{-1}}{|cq+d|^{2}} f(\gamma^{-1}\cdot q)\cdot  R_{n-1}((cq+d)^{-1}),\end{align}
where $\gamma^{-1} = \mtwo{a & b \\ c & d}.$

The $L^2$ norm on $C^\infty(B,\HH^n)$ is defined by 
$$ \|f\|_2 = \left( \int_{B} |f(q)|^2 (1-\rN q)^{n-2} dv \right)^{\frac{1}{2}},\ \ \ dv = dt \wedge dx \wedge dy \wedge dz.$$ 

\textbf{Theorem} \textit{The quaternionic discrete series of $\Sp(1,1)$ with minimal $K$-type $\mathbbm{1}\boxtimes V_n$ occurs in the subspace $C^\infty(B,\HH^n)$ consisting of Fueter-regular functions of finite $L^2$-norm.}

Here a Fueter-regular function $f: B \rar \HH^n$ means one that satisfies $\ol{\partial}_l f = 0$, i.e. such that each individual coordinate function is Fueter-regular. See Theorem \ref{thm2} with a more precise description.

This results allows us to identify the quaternionic analogue of a holomorphic modular form. Fixing a rational form for $G$, and an arithmetic subgroup $\Gamma\subset G(\QQ)$, we say a function $f: B\rar \HH^{n}$ is a \textit{Fueter regular automorphic form}, of weight $n$ and level $\Gamma$, if:
\begin{itemize}
	\item[(1)] $f(\gamma\cdot q) = |cq+d|^2 (cq+d) f(q) \cdot  R_{n-1}(cq+d)$, for all $\gamma=\gamma(a,b,c,d)\in \Gamma$,  
	\item[(2)] $f$ is Fueter-regular,
	\item[(3)] $|f(q)|$ has moderate growth as $|q| \rar 1$.
\end{itemize}

Let $M_n(\Gamma)$, denote the space of Fueter regular automorphic forms of weight $n$ and level $\Gamma$ as above. For $f,g\in M_n(\Gamma)$, define
$$ \omega_{f,g} = g^* Dq f,$$
where $g^*(q) = {}^t\ol{g(\ol{q})}$, and $$ Dq = dx \wedge dy \wedge dz -\qi dt \wedge dy \wedge dz - \qj dt \wedge dx \wedge dz - \qk dt \wedge dx \wedge dy.$$

\textbf{Proposition} \textit{The map $(f,g) \mapsto \omega_{f,g}$ induces one from $M_{k}(\Gamma)\otimes_{\RR} M_k(\Gamma)$ to the space of closed differential $3$-forms $\omega$ on $B$ which satisfy
	$$ \gamma^*(\omega) = R_{n-1}(a+b\ol{q})^*\cdot \omega \cdot R_{n-1}(cq+d),$$
for $\gamma\in \Gamma$.} 

If $\omega_{f,g}=0$, then either $f=0$ or $g=0$, so the map above is non-trivial. See Proposition \ref{formal} in the text for two other differentials associated to $f$ and $g$, in degrees one and two. The fact that the form $\omega_{f,g}$ is \textit{closed} follows from the Fueter regularity of $f$ and $g$.
\section*{Notation}

Throughout, we will write the standard quaternions as
$$ \HH = \{t+ x \qi + y\qj + z \qk : t,x,y,z\in \RR\},$$
where 
$$ \qi^2 = \qj^2 = \qk^2 =\qi\qj\qk=-1.$$
The letters $i,j,k$ are saved for use as indices.  Occasionally it will be 
convenient to write $e_0,e_1,e_2,e_3$ for 
$1,\qi,\qj,\qk$, in that order. The standard involution on $\HH$ is 
\begin{align}q \mapsto \ol{q},\ \ \ q= t + x\qi + y\qj + z\qk,\ \ \ \ol{q}=t - x\qi - y\qj - z\qk.\end{align}The reduced 
norm and trace of $q\in \HH$ are, respectively,
\begin{align} \rN q = q\ol{q},\ \ \ \rT q =q+\ol{q}.\end{align}
For $g = (a_{ij})\in M_n(\HH)$ we write 
\begin{align}g^* = {}^t \ol{g} = (\ol{a_{ji}}).\end{align}
The unit quaternions are denoted by $\HH^1$, and the trace-zero quaternions by $\HH_0$. 

We fix a distinguished embedding $\CC \hra \HH,\ \  a + bi \mapsto a + b\qi$ and identify $\CC$ with its image. Then each $q\in \HH$ is of the form $z+w\qj$ for $z,w\in \CC$. Then standard embedding of $\HH$ into $M_2(\CC)$ is the $\RR$-algebra homomorphism\begin{align}\label{iota}\iota: \HH \hra M_2(\CC),\ \ \ 
\iota(z+w\qj)=\mtwo{z & w \\ -\ol{w} & \ol{z}},\ z,w\in \CC,
\end{align}
so that in particular
\begin{align}\label{emb}\iota(\qi)=\mtwo{i & \\ & -i},\ \ \ 
\iota(\qj)=\mtwo{& 1 \\ -1 & },\ \ \ 
\iota(\qk)=\mtwo{& i \\ i&}.
\end{align} 
The image of $\iota$
is the subalgebra of matrices $A\in M_2(\CC)$ such 
that
$$ A^* A = \det(A) \cdot I.$$
Since
$$ \iota(\ol{q}) = \iota(q)^*,$$
it follows that $\iota$ maps the reduced norm and trace on $\HH$ to the determinant and matrix trace in $M_2(\CC)$, respectively. The map $\iota$ gives an isomorphism of $\HH^1\isomto \SU(2)$.  

All tensor products are taken over $\CC$ if not specified. 

To denote elements of $M_2(\HH)$ without taking up space, we write
$$ \gamma(a,b,c,d) = \mtwo{a & b \\ c & d}.$$
\section{Quaternionic Discrete Series}

\subsection{Gross-Wallach classification}
Let $G$ be a real simple Lie group, with maximal compact $K$, and assume $\rank(G)=\rank(K)$. Then by the work of Harish-Chandra $G$ admits representations in the discrete series. If $\pi$ is such a representation, its restriction to $K$ splits up as
$$ \pi|_{K} = \sum_{\lambda \cleq \mu} m(\mu)\cdot V_{\mu},$$
with $V_{\mu}$ an irreducible representation of $K$, occurring with multiplicity $m(\mu)$. The \textit{minimal $K$-type} of $\pi$ is the equivalence class of the constituent $V_{\lambda}$ with the least highest weight. The representation $\pi$ is determined up to equivalence by its minimal $K$-type.

Assume futhermore that $K$ contains a normal subgroup isomorphic to $\SU(2)$. It follows that in fact $K\simeq \SU(2)\times L$. The following definition is due to Gross-Wallach \cite{GW96}.

\begin{defn}Suppose that $G$ admits discrete series representations, and that $K=\SU(2)\times L$. A discrete series representation $\pi$ of $G$ is called \textit{quaternionic} if its minimal $K$-type $\rho: K \rar \GL(V)$ factors through the projection $K\rar \SU(2)$.
\end{defn}

In other words, the quaternionic discrete series are those whose minimal $K$-types have the form $V\boxtimes \mathbbm{1}$, where $V\in \Rep(\SU(2))$. Let $V_n$ denote the unique irreducible representation of $\SU(2)$ of dimension $n+1$. Gross and Wallach prove that every simple complex Lie group $G_\CC$ has a real form $G$ that admits quaternionic discrete series. Furthermore, for each such $G$ there's a positive integer $d$, easily computable from its root datum, such that the minimal $K$-types of the quaternionic discrete series of $G$ are exactly $V_n\boxtimes \mathbbm{1}$ for $n\geq d$. For example $d=2$ for $G=\Sp(1,1)$. 

We will first recall some standard facts about representations of $\SU(2)$, and fix a model for the representations $V_n$. We then study explicit realizations of the discrete series for $\Sp(1,1)$ using Fueter-regular functions.

\subsection{Representations of $\SU(2)$}

The real Lie group
$$ \SU(2) = \left\{ 
\mtwo{ z & w \\ 
	-\ol{w} & 
	\ol{z}}: z,w\in \CC,\ |z|^2+|w|^2=1\right\}\subset \GL_2(\CC),$$
has Lie algebra
$$ \fraks \fraku(2) = \left\{ \mtwo{ 
	ir & -\ol{u} \\ u & -ir}: u\in \CC,\ r\in \RR\right\}\subset M_2(\CC).$$
The injective map of $\RR$-algebras $\iota: \HH \rar M_2(\CC)$ from (\ref{iota}) restricts to isomorphisms 
$$\HH^1 \isomto \SU(2),\ \ \ \HH_0\isomto \su(2),$$ of Lie groups and Lie algebras, respectively, where the Lie bracket on $\HH_0$ is the commutator of multiplication in $\HH$:
$$ [\qi,\qj] = 2\qk,\ \ \ [\qj,\qk] = 2\qi,\ \ \ [\qk,\qi] = 2\qj.$$
Identifying $\End(\HH_0)$ with $M_3(\RR)$ using the basis $\{\qi,\qj,\qk\}$, the adjoint representation of $\HH_0$ may be expressed as
$$ \ad:\HH_0 \rar M_3(\RR),\ \ \ \ad(a \qi + b\qj + c\qk)=\left(\begin{array}{ccc}  & -2c &  2b\\ 2c& & -2a\\ -2b &2a & 
\end{array}\right).$$
It then follows that, with respect to the basis $\{\qi,\qj,\qk\}$, the matrix of the Killing form on $\HH_0$ is $-8I_3$.

The $\RR$-algebra homomorphism $\iota: \HH \hra M_2(\CC)$ induces an isomorphism $$\iota_\CC: \HH_\CC \isomto M_2(\CC)$$ 
of $\CC$-algebras, which identifies the complexification $\HH_{0,\CC}$ of $\HH_0$ with $\slf_2(\CC)$. The vectors
$$ e = \frac{1}{2}(\qj - i\qk),\ \ \ f = \frac{1}{2}(-\qj-i\qk),\ \ \ h=-i\qi$$
in $\HH_{0,\CC}$ are mapped onto the standard $\slf_2$-triple
$$ X = \mtwo{& 1 \\ & },\ \ Y = \mtwo{&\\ 1 &},\ \ H=\mtwo{1 & \\ & -1}.$$
respectively. The subspace $\frakt_0=\RR h\subset \HH_0$ is then a Cartan subalgebra, with complexification $\frakt =\CC h \subset \HH_{0,\CC}$. We define $\alpha\in\frakt^*$ by
$$ \alpha: \frakt \rar \CC,\ \ \alpha(z\qi)=iz.$$
Then $2\alpha$, $-2\alpha$ are the roots of $\HH_{0,\CC}$, with root vectors $e$, $f$, respectively and root-space decomposition
$$ \HH_{0,\CC} = \CC h \oplus \CC e \oplus \CC f.$$

We choose $2\alpha$ to be the positive simple root. As usual, we let $\frakt_\RR = i \frakt_0 \subset \frakt$, and consider $h\in \frakt_\RR$, $\alpha\in \frakt_\RR^*$. The Killing form restricted to $\frakt_\RR$ induces an isomorphism $\frakt_{\RR}^*\simeq \frakt_\RR$ which maps $\alpha$ to $h_{\alpha} = \frac{1}{8}h \in \frakt_{\RR}$. The fundamental weight associated with $2\alpha\in \frakt_{\RR}^*$ is $h_{\alpha}\in \frakt_\RR=(\frakt_\RR^*)^*$. The weight lattice of $\su(2)$ is then identified with $\ZZ\cdot h_{\alpha} \subset \frakt_{\RR}$. The dominant integral weights are 
$$ \Lambda^+ = \{h_{\alpha}, 2h_{\alpha}, \cdots \}\subset \frakt_{\RR}.$$
They are in bijection with the irreducible representations of $\SU(2)$ via the theorem of highest weight.

The group $\GL_2(\CC)$ acts on polynomials $f(X,Y)\in \CC[X,Y]$ via linear substitutions:
$$ g=\mtwo{a & b \\ c & d},\ \ \ g\cdot f = f(aX+bY,cX+dY).$$
For each $n\geq 0$, let $V_n\subset \CC[X,Y]$ denote the subspace of homogeneous polynomials of degree $n$. The action of $\SU(2)$ on $V_n$ as a subgroup of $\GL_2(\CC)$ makes $V_n$ an irreducible representation, of dimension $n+1$, and highest weight $nh_{\alpha}$, with highest weight vector $X^n$.

The representation $V_1$ constructed above coincides with the \textit{standard representation} of $\SU(2)\subset \GL_2(\CC)$ acting on $\CC^2$. On the other hand, identifying a complex number $x+iy$ with $x+\qi y\in \HH$, the quaternions $\HH$ become a complex vector space via right-multiplication by $\CC$, and a complex representation of $\HH^1$ via left-multiplication. Thus $\SU(2)$ acts on $\HH$ via $\iota: \HH^1 \isomto \SU(2)$. There is an equivalence of representations
\begin{align}\label{V1H}V_1 \rar \HH,\ \ \ aX + bY\mapsto a - b\qj,\end{align}
identifying $\HH$ with the standard representation of $\SU(2)$.

For 
$$J=\mtwo{& -1 \\ 1 &},$$
and $g\in \GL_2(\CC)$, we have 
$$ Jg J^{-1} = \det(g) \cdot {}^t g^{-1}.$$
Thus the standard representation $V_1$ of $\GL_2(\CC)$ is isomorphic to $\det\otimes V_1^*$. It follows that, since $\det|_{\SU(2)}=1$, the standard representation of $\SU(2)$ is self-dual.  If $X^*$, $Y^*\in V_1^*$ denote the dual basis of $X$, $Y\in V_1$, there is an explicit isomorphism of representations
\begin{align}\label{V1dual} V_1^* \rar V_1,\ \  aX^* + bY^* \mapsto bX - aY.\end{align}
Note that, since $V_n = \Sym^{n}(V_1)$, all representations of $\SU(2)$ are self-dual.

Multiplication of polynomials induces a surjective $\SU(2)$-equivariant map 
\begin{align}\label{Vmn} V_m \otimes V_n \rar V_{m+n},\ \ \ f\otimes g \mapsto fg.\end{align}
When $m=1$ and $n\geq 1$, the kernel of this map is isomorphic to 
$V_{n-1}$, so that $V_1\otimes V_{n} \simeq V_{n-1}\oplus V_{n+1}.$ Since $V_1$ is self-dual, we then have $V_1^*\otimes V_n \simeq V_{n-1}\oplus V_{n+1}$. We shall make this isomorphism explicit.

Let
\begin{align}\partial_X,\ \partial_Y:\CC[X,Y] \rar \CC[X,Y] \end{align}
denote differentiation with respect to $X$ and $Y$, respectively. Then we have $\partial_X, \partial_Y: V_k \rar V_{k-1}$ for each $k\geq 1$. Identifying $V_0$ with $\CC$, the maps $\partial_X,\ \partial_Y: V_1 \rar \CC$ coincide with $X^*, Y^*\in V_1^*$ dual to $X$ and $Y$. For $f\in \CC[X,Y]$, we write
\begin{align} f_X = \partial_X f,\ \ \ f_Y = \partial_Y f.\end{align}
Then we have $\SU(2)$-equivariant maps
\begin{align} 
&P^+: V_1^*\otimes V_k \rar V_{k+1},\ \ \ (a \partial_X + b \partial_Y) \otimes f \mapsto (bX-aY)f, & &  \\
\label{Pmin} &P^-: V_1^*\otimes V_k \rar V_{k-1},\ \ \ (a \partial_X + b \partial_Y )\otimes f \mapsto a f_X + b f_Y.& & \end{align}
The map $P^+$ is the composition of (\ref{Vmn}) for $(m,n)=(1,k)$ with $V_1^*\isomto V_1$ as in (\ref{V1dual}). The map
\begin{align} P^+\oplus P^- : V_1^*\otimes V_k \lra V_{k+1} \oplus V_{k-1}\end{align}
is an isomorphism of $\SU(2)$-representations

\section{Regular Functions on the Quaternionic Ball}

\subsection{The Structure of $\Sp(1,1)$}

\subsubsection{Subgroups and Lie subalgebras} 
Let 
\begin{align}G&= \left\{\mtwo {a & b \\ c & d}\in M_2(\HH) : \rN a = \rN c + 1,\ \ 
\rN b = \rN 
d -1,\ \ \ol{a}b=\ol{c}d\right\},\\
\label{Sp11K} K &= \left\{\mtwo{a & \\ & d}: a,d\in \HH^1\right\} \simeq \HH^1\times \HH^1,\\
H &= \left\{\mtwo{u & \\ & u^{-1}}: u\in \HH^1\right\}
\simeq\HH^1.\end{align}

The defining condition for $G$ can also be written for $g=g(a,b,c,d)\in M_2(\HH)$ as
$$g \in G\iff\mtwo{ a & b \\ c & d }\left(\begin{array}{rr}
\ol{a} & -\ol{c} \\ -\ol{b} & \ol{d}\end{array}\right)= \mtwo{1 & \\   & 1}.$$

$G=\Sp(1,1)$ is a non-quasi-split inner form of $\Sp(4)$, the subgroup $K\subset G$ is maximal compact, and $H\subset K$ is a compact Cartan subgroup of $G$ contained in $K$. In particular, $\rank(K)=\rank(G)$, and $G$ admits representations in the discrete series. Since $K=\HH^1\times \HH^1$, $G$ admits quaternionic discrete series in two families, corresponding to each $\HH^1$ factor. We mainly choose the second $\HH^1$ factor, and consider quaternionic discrete series with minimal $K$-types of the form $\mathbbm{1}\boxtimes V_n$. The minimal $K$-types that do occur are exactly those with $n\geq 2$.

The Lie algebras of $G$, $K$, and $H$ are, respectively,
\begin{align} \frakg_0 &= \Lie(G) = \left\{ \mtwo{ a & b \\ \ol{b} & d }: 
a,d\in \HH_0, b\in \HH\right\}.\\\frakk_0 &=\Lie(K)= \left\{\mtwo{ a & \\ & d}: a,b\in \HH,\ \Tr 
a = \Tr d = 
0\right\},\\
\frakh_0 &= \Lie(H)=\left\{ \mtwo { r\qi & \\ & s\qi }: r,s\in \RR \right\}.
\end{align}
Then $\frakh_0\subset \frakk_0$ is a Cartan subalgebra of $\frakg_0$. 
Let $S_1$ be the ordered set of $\frakk_0$ given by 
\begin{align} \mtwo{\qi & \\ & }, \mtwo{\qj & \\ & }, \mtwo{\qk & \\ & }, \mtwo{ & \\ & \qi }, \mtwo{ & \\ & \qj }, \mtwo{ & \\ & \qk}\end{align}
and $S_2$ the set
$$ \mtwo{ & 1\\ 1& }, \mtwo{ & \qi\\ -\qi & }, \mtwo{ & \qj \\-\qj & }, \mtwo{ &\qk \\ -\qk& }.$$
The matrix of the Killing form on $\frakg_0$ with respect to $S = S_1\cup S_2$ is 
$$ \mtwo{-12 I_6 & \\ & 24 I_4}.$$
In particular, the Killing form induces an orthogonal decomposition $\frakg_0 = \frakk_0 \oplus \frakp_0$, where
\begin{align}\label{p0}\frakp_0 = \left\{ \mtwo{& b \\ \ol{b} &}: b\in \HH\right\}
\end{align}
has orthogonal basis $S_2$. The adjoint representation of $\frakg_0$ restricts to an action of $\frakk_0$ on $\frakp_0$.

Let $\frakg$, $\frakk$, $\frakh$, $\frakp$ denote the complexifications of $\frakg_0$, $\frakk_0$, $\frakh_0$, $\frakp_0$, respectively. Using the notation
$$ \HH_\CC =\{q=t + x\qi + y\qj + z\qk: t,x,y,z\in \CC\},\ \ \ 
\ol{q} := t - x\qi - y \qj - z\qk,$$
we have
$$ \frakg = \left\{ \mtwo {a & b \\ \ol{b} & d}: 
a,b,c,d\in \HH_\CC,\ \ a+\ol{a}=d+\ol{d}=0\right\}.$$
Define $\alpha,\beta\in \frakh^*$ by
\begin{align}\alpha\left(\mtwo {z \qi &\\ & w\qi}\right) = iz,\ \ \ \beta\left(\mtwo {z \qi &\\ & w\qi}\right) = iw
\end{align}
The root system of $(\frakg,\frakh)$ is then
\begin{align} \Phi = \{ \pm 2\alpha, \pm 2\beta, \pm \alpha \pm \beta\}.\end{align}

We put 
$$\frakh_\RR = i \frakh_0 \subset \frakh, \ \ \ \frakh_\RR^* = 
\Hom(\frakh_\RR,\RR).$$
and note that $\frakh_\RR^*=i\frakh_0^*$ canonically. Since $\alpha$, $\beta$ take real values on $\frakh_\RR$, we may consider $\Phi\subset \frakh_\RR^*$. The $\CC$-bilinear extension of the Killing form to $\frakh$ restricts to a real-valued inner product on $\frakh_\RR$. Letting
\begin{align} H_\alpha = \frac{-1}{12}\mtwo{i\qi & \\ & },\ \ \ H_\beta = \frac{-1}{12}\mtwo{& \\ & i\qi},\end{align}
and
$$H_{\mu}=r H_{\alpha} + s H_{\beta},\ \ \ \text{ for }\mu = r\alpha + s \beta,\ r,s\in \ZZ,$$
we have $H_{\mu}\in \frakh_\RR$, and if $\mu\in \Phi$, then 
$$ B(H_{\mu},X) = \mu(X),\ \ \ X\in \frakh.$$

The root spaces $\frakg_{\mu}$ are contained in $\frakk$ for the compact roots
\begin{align} \Phi_{\frakk} = \{\pm 2\alpha , \pm 2\beta\},\end{align}
and spanned by root vectors
\begin{align}
\begin{split}
E_{2\alpha} = \frac{1}{\sqrt{-24}}\mtwo{\qj - i\qk & \\&  },\ E_{-2\alpha}=\frac{1}{\sqrt{-24}}\mtwo{\qj +i\qk & \\ 
	& },\\\ E_{2\beta} = \frac{1}{\sqrt{-24}}\mtwo{ & \\ & \qj-i\qk},\ E_{-2\beta} = \frac{1}{\sqrt{-24}}\mtwo{ & \\ & 
	\qj+i\qk}.
\end{split}
\end{align}
The non-compact roots
\begin{align}\Phi_{\frakp} = 
\{ \pm \alpha \pm \beta\},\end{align}
correspond to root spaces spanned by
\begin{align} 
\begin{split}E_{\alpha-\beta} = \frac{-1}{\sqrt{-48}}\mtwo{ & 1-i\qi \\ 1+i\qi & },\ E_{\alpha+\beta} = 
\frac{-1}{\sqrt{-48}}\mtwo{& \qj-i\qk\\ -\qj+i\qk},\\
E_{-\alpha+\beta}=\frac{1}{\sqrt{-48}}\mtwo{& 1+i\qi \\ 1-i\qi&}, E_{-\alpha-\beta}=\frac{1}{\sqrt{-48}}\mtwo{& 
	\qj+i\qk\\ -\qj-i\qk&}.
\end{split}
\end{align}
The choices above are normalized so that
$$ B(E_{\mu},E_{-\nu}) = \delta_{\mu,\nu},\ \ \ [E_{\mu},E_{-\mu}]=H_{\mu},\ \ \text{ for }\mu,\nu \in \Phi.$$

Letting $\frakg_\mu = \CC E_\mu$ for each $\mu\in \Phi$, we have 
$$ \frakk = \frakh \oplus \bigoplus_{\mu\in \Phi_{\frakk}} \frakg_\mu,\ \ \ \frakp = 
\bigoplus_{\mu\in \Phi_{\frakp}} \frakg_{\mu}.$$

The lattice $\Lambda \subset \frakh_\RR^*$ is spanned by $\alpha$, $\beta$. We choose the positive simple roots to be
$$ \Delta = \{2\alpha,2\beta, \alpha+\beta, \alpha-\beta\},$$
and denote the positive compact and non-compact simple roots, respectively, by
$$ \Delta_{\frakk} = \{2\alpha, 2\beta\},\ \ \ \Delta_{\frakp} = \{\alpha+\beta,\alpha-\beta\}.$$

The weight lattice of $\frakk$ is $\ZZ h_\alpha + \ZZ h_\beta$, the sum of two weight lattices, one for each copy of $\SU(2)$. The dominant integral weights with respect to $\Delta_{\frakk}$ are then $mh_{\alpha} + nh_{\beta}$, $m,n\geq 0$, corresponding to the irreducible representation $V_{m}\boxtimes V_n\in \Rep(K)$.

\subsubsection{The symmetric space}

The group $G$ acts on $\HH$ via fractional linear transformations:
$$ \mtwo{a & b \\ c &d}\cdot q = (a q + b)(c q 
+ d)^{-1}.$$

\begin{lemma}\label{Ngamma}	For $g=\mtwo{a & b\\ c & d} \in \Sp(1,1)$, 
	$q\in \HH$, we 
	have $$ 1-\rN(g \cdot q) = \frac{1-\rN q}{\rN(cq + d)}.$$
	In paricular, $\rN q < 1$ if and only if $\rN(g\cdot q)<1$.
\end{lemma}
\begin{proof}
	For $g = \mtwo{a& b \\ c & d}\in \Sp(1,1)$, 
	we have
	$$ \rN(cq+d) = \rN q \rN c + \rN d + 
	\rT(cq\ol{d}),\ \ \  \rN(aq+b) = \rN q \rN a + \rN b + \rT(aq\ol{b}).$$
	In general $\Tr(aq\ol{b}) = \Tr(\ol{b}aq)$ and 
	$\Tr(cq\ol{d})=\Tr(\ol{d}cq)$ by 
	the properties of trace. Since $\gamma\in \Sp(1,1)$, we also have
	$$ \ol{b}a=\ol{d}c,\ \ \ \rN c = \rN a - 1,\ \ \ \rN d = \rN b + 1.$$
	From the first relation we get $\Tr(\ol{b}aq) =\Tr(cq\ol{d})$. Using that 
	and the other two relations,
	$$ \rN (cq+d) - \rN (aq+b) = \rN q(\rN a - 1) +  \rN b + 1 - \rN q \rN a - 
	\rN 
	b = 1 -\rN q.$$
	The lemma then follows from 
	$$ N(\gamma \cdot q) = 	\frac{N(aq+b)}{N(cq+d)}.$$
\end{proof}
It follows that the action of $G$ stabilizes the quaternionic ball
\begin{align}B = \{q\in \HH : \rN q< 1\},\end{align}
and hence the map
\begin{align}\label{GB}G\rar B,\ \ \ g=\mtwo{a & b \\ c & 
	d} \mapsto g\cdot 0 = bd^{-1}\end{align}
is well-defined. In fact the map is surjective, with a right-inverse 
\begin{equation}\label{sigma}\sigma: B \rar G,\ \ \ \sigma(q) = 
\frac{1}{\sqrt{1-\rN q}} \mtwo{ 1 & q \\ \ol{q} & 1 }.
\end{equation}
Moreover, the stabilizer of $0\in B$ is the maximal compact subgroup $K$. Therefore (\ref{GB}) induces a bijection $G/K \rar B$, identifying $B$ with the symmetric space of $K$. We note that $\sigma$ satisfies
\begin{equation}\label{sigmarels} \sigma(q)^{-1} = 
\sigma(-q).\end{equation}

\subsubsection{Induced Representations}

For $V\in \Rep(K)$, we have the induced (smooth) representation
$$\Ind_{K}^G(V) = \{F: G \rar V\text{ smooth }: k\cdot F(g)=F(gk^{-1})\text{ for all }k\in K\}.$$

We fix the map $\sigma: B\rar G$ to be as in (\ref{sigma}), and correspondingly, we define
\begin{align} \jmath(\gamma,x): G\times B \rar K,\ \ \ \jmath(\gamma,x) = \sigma(\gamma\cdot x)^{-1} 
\gamma \sigma(x).\end{align}

\begin{lemma}For $\gamma\in G$ and $q\in B$ as above we have
	$$ \jmath(\gamma,q) = \frac{1}{|cq+d|}\mtwo{a+b\ol{q}& \\ & cq+d}.$$
\end{lemma}
\begin{proof}Using the property $\sigma(q)^{-1} = \sigma(-q)$, we have
	\begin{align*} \jmath(\gamma,q) &= \sigma(-\gamma \cdot q) \gamma \sigma(q) 
	=\frac{1}{\sqrt{1-\rN q}\sqrt{1-\rN (\gamma \cdot q)}}\mtwo{1 & -\gamma\cdot q \\ 
		-\ol{\gamma\cdot q} & 1}\mtwo{a & b \\ c &d}\mtwo{1 & q \\ \ol{q}& 
		1}\\
	&=\frac{1} 
	{\sqrt{1-\rN q}\sqrt{1-\rN(\gamma \cdot q)}}\mtwo{a+b\ol{q} - (\gamma\cdot 
		q)(c+d\ol{q}) & aq+b - (\gamma\cdot q)(cq+d)\\ -\ol{\gamma\cdot 
			q}(a+b\ol{q}) + 
		c + d \ol{q}& -\ol{\gamma\cdot q} (aq+b) + cq+d}.
	\end{align*}
	The off-diagonals are automatically zero since $\jmath(q,\gamma)\in K$. Indeed, the 
	top-right is zero by definition. Using
	$$ q = \mtwo{\ol{a} & -\ol{c} \\ -\ol{b} & \ol{d}}\cdot (\gamma\cdot q) = 
	(\ol{a}(\gamma\cdot q) - \ol{c})(-\ol{b} (\gamma\cdot q) + \ol{d})^{-1}$$
	we obtain
	$$ -q \ol{b}(\gamma\cdot q) + q \ol{d} = \ol{a}(\gamma\cdot q) - \ol{c}\so 
	c + 
	d \ol{q} = \ol{\gamma\cdot q}(a + b\ol{q}),$$
	which implies the bottom left is also zero. Now using the same relations on the diagonal entries we obtain 
	$$ \jmath(\gamma,q) = \left(\frac{1-\rN(\gamma\cdot q)}{1-\rN 
		q}\right)^{1/2}\mtwo{ 
		a+  b\ol{q} & \\ & cq+d} = \frac{1}{|cq + d|} \mtwo{ 
		a+  b\ol{q} & \\ & cq+d},$$
	where the second equality follows from Lemma \ref{Ngamma}. 
\end{proof}	

Let $V\in \Rep(K)$ have dimension $n+1$, and let $C^\infty(B,V)$ denote the space of smooth maps $f: B \rar V$. We have an isomorphism of vector spaces
\begin{align}\label{trivia} \Ind_K^G(V) \rar C^\infty(B,V),\ \ \ F \mapsto f,\ \ f(q) = F(\sigma(q)) \cdot  | 1-\rN q |^{n+2}.\end{align}
We let $\gamma\in G$ act on $f\in C^\infty(B,V)$ by transport of structure, so that the above is upgraded to an isomorphism of representations.

Now let $V_n$ be the representation of $\HH^\times$ on $\CC[X,Y]_{n}$. Let $(\rho_{\tau},V_n), (\lambda_{\tau}, V_n)\in \Rep(K)$ be defined by
\begin{align}\label{rhotau}\rho_{\tau}(a,d) = \tau(d),\ \ \ \lambda_{\tau}(a,d) = \tau(a),\ \  a,d\in \HH^1.\end{align}
Let $C^\infty(B,\rho_{\tau})$, resp. $C^\infty(B,\lambda_{\tau})$, denote the action of $G$ on $C^\infty(B,V)$ obtained from $\Ind_{K}^G(\rho_{\tau},V)$, resp. $\Ind_{K}^G(\lambda_{\tau},V)$, via the map $(\ref{trivia})$. For $\gamma\in G$, if we write $\gamma^{-1} = \gamma(a,b,c,d)$, then $\gamma$ acts on $f\in C^\infty(B,\rho_{\tau})$ and $g\in C^\infty(B,\lambda_{\tau})$ by
\begin{align}\label{Gaction} (\gamma\cdot f)(q) &= |cq+d|^2\tau\left((cq+d)^{-1}\right) f((aq+b)(cq+d)^{-1}),\\
(\gamma \cdot g)(q) & = |a+b\ol{q}|^2 \tau\left((a+b\ol{q})^{-1}\right) g((aq+b)(cq+d)^{-1}).\end{align}
In particular, if $\Gamma \subset G$ is a subgroup, then $f$, respectively $g$, is $\Gamma$-invariant if and only if 
\begin{align} f(\gamma\cdot q) &= \frac{1}{|cq+d|^2}\tau(cq+d) f(q),
\end{align}
respectively,
\begin{align}g(\gamma\cdot q) &= \frac{1}{|a+b\ol{q}|}\tau(a+b\ol{q}) g(q),\end{align} Note that $|cq+d|=|a+b\ol{q}|$ if $g\in \Sp(1,1)$.

For $\gamma\in M_2(\HH)$, we write
\begin{align} \gamma^\dag = R\gamma R =  \mtwo{d & c \\ b & a},\ \ \ \text{ where }\ \ \ R = \mtwo{& 1 \\ 1 & },\ \ \ \gamma = \mtwo{a & b\\c& d}.\end{align}
Then $\gamma \mapsto \gamma^{\dag}$ is an order-two automorphism of the ring $M_2(\HH)$, that preserves both $G$ and $K$, and induces a vector space isomorphism
\begin{align}C^\infty(G,V) \rar C^\infty(G,V),\ \ \ F \mapsto F^\dag,\ \ \ F^{\dag}(x) = F(x^\dag).\end{align}
For $(\tau,V)\in \Rep(K)$, we define
\begin{align} (\tau^{\dag},V^{\dag})\in \Rep(K),\ \ \  
V^{\dag} =V,\ \ \ \tau^\dag(k) = \tau(k^\dag).\end{align}

Let $C^\infty(B,V^\dag)$ denote the $G$-representation obtained from $\Ind_{K}^G(\tau^{\dag},V^{\dag})$ via (\ref{trivia}). Now we define a vector space isomorphism
\begin{align} C^\infty(B,V) \rar C^\infty(B,V^\dag), \ \ f \mapsto f^{\dag},\ \ f^{\dag}(q) = f(\ol{q}).\end{align}

\begin{lemma}\label{daglem}
There is a commutative diagram of vector space isomorphisms
$$ \xymatrix{\Ind_{K}^G(\rho,V)\ar[d] \ar[r]^{F \mapsto F^\dag} &\Ind_K^G(\tau^\dag,V^{\dag}) \ar[d]\\
	C^\infty(B,V) \ar[r]_{f\rar f^\dag} &C^\infty(B,V^{\dag}),}$$
where the vertical arrows are $G$-equivariant, given by (\ref{trivia}), and the top and bottom arrows satisfy
$$ (\gamma\cdot f)^\dag = \gamma^\dag \cdot f^{\dag},\ \ \ (\gamma\cdot F)^{\dag} = \gamma^\dag \cdot F^{\dag},\ \ \ \gamma\in G.$$
\end{lemma}
\begin{proof} 
	For $F\in \Ind_K^G(\tau,V)$, we have
	$$ F^\dag(xk^{-1}) = F(x^{\dag} (k^\dag)^{-1})= \tau^\dag(k) F^{\dag}(x),$$
	which shows the top arrow is well-defined. The intertwining property of $F\mapsto F^{\dag}$ follows from the fact that $(\gamma^{\dag})^{-1} =(\gamma^{-1})^{\dag}$. Since the vertical arrows are $G$-equivariant isomorphisms by definition, there's a unique isomorphism $f \mapsto f^{\dag}$ satisfying the rest of the claim. That it is the map $f^{\dag}(q) = f(\ol{q})$ follows from the fact that $\sigma(q)^{\dag}=\sigma(\ol{q})$ for the map $\sigma: B\rar G$ in (\ref{sigma}).  
\end{proof}

\begin{prop}\label{dagprop}Let $f\in C^\infty(B,V_n)$, and $\Gamma\subset G$ be a subgroup. Then
	$$ f(\gamma\cdot q) = |cq+d|^2\rho_m\left(cq+d\right)\cdot f(q) \text{\ \ \ for all\ \ } \gamma=\mtwo{a & b \\c & d}\in \Gamma,\ q\in B,$$
	if and only if $f^{\dag}(q) =f(\ol{q})$ satisfies
	$$f^\dag(\gamma\cdot q)= |a+b\ol{q}|^2\rho_n(a+b\ol{q}) f^{\dag}(q)\ \ \text{\ \ \ for all\ \ }\gamma=\mtwo{a & b \\ c & d}\in \Gamma,\ q\in B.$$
\end{prop}
\begin{proof} Let $(\tau,V)\in \Rep(K)$ be the representation $V_0\boxtimes V_n$, so that $\rho = \rho_{\tau}$. Then the first statement holds if and only if $f\in C^{\infty}(B,\rho_{\tau})^{\Gamma}$, and the second if and only if $f^{\dag}\in C^{\infty}(B,\lambda_{\tau})^{\Gamma}$. By Lemma \ref{daglem}, $\gamma \cdot f = f$ if and only if $\gamma^{\dag} \cdot f^{\dag} = f^{\dag}$. 
\end{proof}

\subsection{Dirac Operators}

Let $V$ be a representation of $K=\HH^1\times \HH^1$. Recall that for a smooth function $f: G \rar V$, and $X\in \frakp_0$, the Lie derivative of $f$ along $X$ is
\begin{align}\label{Xdiff} X\cdot f: G \rar V,\ \ \ (X\cdot f)(g) = \left.\frac{d}{dt}f(g\exp(tX))\right|_{t=0}.\end{align}
For $f\in C^\infty(G,V)$, we put
\begin{align} df: G \rar \frakp	\otimes V,\ \ \ f = \sum_{\beta\in \Phi_{\frakp}} E_{-\beta} \otimes E_{\beta}\cdot f.\end{align}
Identifying $\frakp_0$ with $\HH$ as in (\ref{p0}), the adjoint action of $K$ on $\frakp_0$ is 
\begin{align}\label{p0act} (a,b)\cdot x = axb^{-1},\ \ x\in \frakp_0=\HH,\ \ \ a,b\in \HH^1.\end{align}
Then $K$ acts similarly on the complexification $\frakp=\HH_{\CC}$.

\begin{lemma}\label{pV11} The map $\frakp = \HH_{\CC} \rar V_1\boxtimes V_1^*$ sending $1,\qi,\qj,\qk$, respectively, to	
		$$ X\otimes \partial_X + Y \otimes \partial_Y,\ \ iX\otimes \partial_X - i Y\otimes \partial_Y,\ \ X\otimes \partial_Y - Y\otimes \partial_X,\ \ iX\otimes \partial_Y + i Y \otimes \partial_X,$$
is an isomorphism of $K$-representations.
\end{lemma}
\begin{proof}
The map $\iota_{\CC}: \HH_{\CC} \rar M_2(\CC)$ identifies the action of $K$ on $\frakp$ given by (\ref{p0act}), with that of $\SU(2)\times \SU(2)$ on $M_2(\CC)$ given by
$$ (A,B)\cdot X = AXB^{-1},\ \ X\in M_2(\CC),\ \ \ (A,B)\in \SU(2)\times \SU(2),$$
where $X=\iota(x)$, $A=\iota_{\CC}(a)$, $B=\iota_\CC(b)$, for $a,b\in \HH^1$. Composing with 
$$ M_2(\CC) \rar V_1\boxtimes V_1^*,\ \ \ \mtwo{a & b \\ c & d} \mapsto a X \otimes \partial_X + b X\otimes \partial_Y + c Y\otimes \partial_X + d Y\otimes \partial_Y$$
we obtain the desired map.
\end{proof}

The representation of $K=\HH^1\times \HH^1$ on $V_n\subset \CC[X,Y]$ obtained by projection onto the second $\HH^1$ factor, can be written as $V_0\boxtimes V_n$. Using $V_1 \isomto \HH$ fixed in (\ref{V1H}, we write
\begin{align}W_n = \HH\otimes V_{n} = V_1\boxtimes V_{n}\in \Rep(K).
\end{align}
Since
$$ V_1^*\otimes V_n \cong V_{n-1}\boxtimes V_{n+1},$$
we have isomorphisms of $K$-representations
$$ (V_1\boxtimes V_1^*)\otimes (V_0 \boxtimes V_n) \cong V_1 \boxtimes (V_{n+1}\oplus V_{n-1})\cong W_{n+1} \oplus W_{n-1}.$$
Then writing $V_n$ for $V_0\boxtimes V_n$, and using Lemma (\ref{pV11}) we obtain
$$ \frakp\otimes V_n \isomto W_{n+1}\oplus W_{n-1}.$$
Let
\begin{align}\label{proj}\pi^{-}_n: \frakp\otimes V_n \rar W_{n-1}\end{align}
denote the canonical projection map. Then the Dirac operator $\scrD$ of W. Schmidt \cite{Sch89} associated with $V_n$ can be identified with
\begin{align}\label{Dirac} \scrD: C^\infty(G,V_{n}) \rar C^{\infty}(G,W_{n-1}),\ \ \scrD(f)= \pi_{n}^{-}(df).\end{align}

Since $\CC[X,Y]=\bigoplus_{n\geq 0} V_n$, we have $\HH[X,Y]=\bigoplus_{n\geq 0} W_n$. By a smooth function $f: G \rar \HH[X,Y]$ we mean one whose image is contained in a finite-dimensional subspace of $\HH[X,Y]$, such that, with respect to a real basis, each coordinate function $G \rar \RR$ is smooth. Let $C^\infty(G,\HH[X,Y])$ be the space of all such functions. Then $\frakg_0$ acts on this space by Lie differentiation (\ref{Xdiff}). We define an operator $\ol{\partial}$ acting on $\CC^\infty(G,\HH[X,Y])$ by
\begin{align}\label{dG} \ol{\partial} = E_0 + \qi E_1 + \qj E_2 + \qk E_3.\end{align}

\begin{prop}\label{dXYprop}The following diagram is commutative:
	$$ \xymatrix{
		C^{\infty}(G,V_n) \ar[d]_d \ar[r]^-{\frac{1}{24}(\partial_X - \qj \partial_Y)} &C^{\infty}(G,W_{n-1}) \ar[d]^{\ol{\partial}}\\
		C^{\infty}(G,\frakp\otimes V_n) \ar[r]^-{\pi^-_n} & C^\infty(G,W_{n-1}).}$$
In particular, $\scrD(f)=0$ if and only if $\ol{\partial}(f_X - \qj f_Y)=0$.
\end{prop}
\begin{proof} Let $\{v_i\}$, $\{v_i'\}$ be a pair of bases for $\frakp$ such that the dual of the first is identified with the second under the isomorphism $\frakp \rar \frakp^*$ provided by the Killing form. Then
	$$ e= \sum_{i} v_i \otimes v_i'\in \frakp\otimes \frakp$$
	is the image of the identity map $\mathbbm{1}_{\frakp}\in \End(\frakp)$ under the composition of isomorphisms 
	$$\End(\frakp)\cong \frakp\otimes \frakp^* \simeq \frakp \otimes \frakp.$$
	In particular, $e$ does not depend on the choice of $\{v_i\}$, $\{v_i'\}$, and we have 
	$$ e = \sum_{\mu\in \Phi_{\frakp}} E_{-\mu}\otimes E_{\mu}  = \frac{1}{24} \sum_{i=0}^3 E_i\otimes E_i,$$
	where $E_i$ are the image of $e_i\in \HH_{\CC}$ in $\frakp$ under the map of Lemma \ref{pV11}. Writing
	$$ \delta: \frakp \otimes C^\infty(G,V) \rar C^\infty(G,V),\ \ \ \delta(X\otimes f) = X\cdot f,$$ we have
	$$ df = (\mathbbm{1}_{\frakp}\otimes \delta) (e\otimes f) = \frac{1}{24}\sum_{i=0}^3 E_i\otimes E_i f.$$
Then by Lemma \ref{pV11}, we can write
	\begin{align*} 24(df) &=(X\otimes \partial_X + Y \otimes \partial_Y)\otimes E_0 f + (iX\otimes \partial_X - i Y\otimes \partial_Y)\otimes E_1 f\\
	+ &(X\otimes \partial_Y - Y\otimes \partial_X)\otimes E_2 f + (iX\otimes \partial_Y + i Y \otimes \partial_X)\otimes E_3f,\end{align*}
as a map $G \rar V_1\otimes V_1^*\otimes V_{n-1}$. Writing $P_n^{-}: V_1^*\otimes V_n \rar V_{n-1}$ for the projection map, we have
	\begin{align*} &24(\mathbbm{1}_{V_1}\otimes P_{n}^-)(df)=X\otimes (E_0 + iE_1)f_X + X\otimes (E_2 +i E_3 )f_Y + Y\otimes (-E_2 + i E_3) f_X +  Y\otimes (E_0  - i E_1 )f_Y\end{align*}
	using the fact that the operators $\partial_X$ and $\partial_Y$ commute with Lie differentiation.
	Composing with the isomorphism $V_1\otimes V_{n-1} \rar W_{n-1}$ that sends $X$, $Y\in V_1$ to $1$, $\qj\in \HH$, respectively, we have
	\begin{align*} 24\pi_{n}^{-}(df)&= (E_0 + \qi E_1)f_X + (E_2+\qi E_3)f_Y -\qj (-E_2 + \qi E_3) f_X - \qj (E_0 - \qi E_1) f_Y\\
	&= (E_0 + \qi E_1 + \qj E_2 + \qk E_3) (f_X - \qj f_Y) = \ol{\partial}(\partial_X - \qj \partial_Y) f.
	\end{align*}
\end{proof}

Now we repeat the above for functions on $B$ rather than $G$. For $f: B \rar V_n$ smooth, we put
$$ df = \sum_{i=0}^3 E_i\otimes \pard{f}{x_i}.$$
The following Dirac operator was considered in \cite{LiuZhang}
\begin{align}\label{LZDirac} \calD(f)= \pi_n^{-}(df),\end{align}
with $\pi_{n}^-$ the same as in (\ref{proj}). The same result then holds:

\begin{prop}\label{d2prop}The following diagram is commutative:
		$$ \xymatrix{
			C^{\infty}(B,V_n) \ar[d]_d \ar[r]^-{\frac{1}{24}(\partial_X - \qj \partial_Y)} &C^{\infty}(B,W_{n-1}) \ar[d]^{\ol{\partial}}\\
			C^{\infty}(B,\frakp\otimes V_n) \ar[r]^-{\pi^-_n} & C^\infty(B,W_{n-1}).}$$
	In particular, $\calD(f)=0$ if and only if $\ol{\partial}(f_X - \qj f_Y)=0$.
\end{prop}
\begin{proof}In the proof of Proposition \ref{dXYprop}, the domain only entered the picture via the fact that the operators $\partial_X$ and $\partial_Y$ commute with Lie differentiation. Since they also commute with the operators $\pard{}{x_i}$, the exact same proof works.
\end{proof}

Recall that the operator $\ol{\partial}$ on $C^n(B,W_n)$ is the (left-sided) Fueter operator  
\begin{align} \ol{\partial}_l = \frac{d}{dt} + \qi \frac{d}{dx} + \qj \frac{d}{dy} + \qk\frac{d}{dz}.\end{align}
A function $h: \HH \rar \HH$ is called \textit{Fueter regular} or just \textit{regular} if $\ol{\partial}_l h = 0$. (A standard reference is \cite{Sud79}. See also the appendix for more information.)

\begin{thm}\label{thm1}For $F\in C^\infty(G,V_n)$, let $f(q)=(1-\rN q)^{n-2} F(\sigma q)\in C^\infty(B,V_n)$, and write 
	$$ f_X - \qj f_Y = \sum_{i=0}^{n-1} \phi_{i+1} X^i Y^{n-i},\ \ \ \phi_i: B \rar \HH.$$
Then $\scrD(F)=0$ if and only if each $\phi_i$ is Fueter-regular, where $\scrD$ is the Dirac operator (\ref{Dirac}).
\end{thm}
\begin{proof}
The equivalence
$$ \scrD(F) = 0 \iff \calD(f)=0$$
is proven in \cite[Prop. 3.6]{LiuZhang}.
By Proposition \ref{d2prop}, we have
$$ \calD(f)=0\iff \ol{\partial}_l(\partial_X f -\qj\partial_Y f)=0\iff \ol{\partial}_l \phi_i=0,\ \ i=1,\cdots,n.$$
\end{proof}

Let $\HH^n$ be a considered a complex vector space via right multiplication by $\CC\subset \HH$, considered as a space of quaternionic row vectors, with hermitian inner product $(q_1,q_2) = \ol{q_1} \cdot {}^t q_2$. 

Define a $G$-action on $C^\infty(B,\HH^{n})$ by
\begin{align}\label{GRact} (\gamma \cdot h)(q) = |cq+d|^{-2}(cq+d)^{-1} h(\gamma^{-1}q) \cdot {}^t R_{n-1}(cq+d),\ \ \ u=\gamma^{-1},q)^{-1},\end{align}
where $R_{n-1}(u)\in \GL_{n}(\CC)$ is the matrix of $\rho_{n-1}(u)$ acting on $V_{n-1}$ with respect to the basis $X^{n-1-i}Y^i$, $i=0,\cdots, n-1$. Fix a hermitian inner product $[\ ,\ ]$ on $\CC^n$ with respect to which the action of $\HH^1$ via $R_{n-1}$ is unitary, and extend it to $\HH^n =\HH\otimes \CC^n$.

The $L^2$-norm on $C^\infty(B,\HH^n)$ is given by
$$ \|h\|_2 = \left(\int_{B} [h(q),h(q)] 	 (1-\rN q)^{n+2} d\mu\right)^{\frac{1}{2}},$$
where $d\mu$ is the Lebesgue measure on $B\subset \HH$. 

Let $\scrR(B,\HH^n)$ denote the subspace of $C^\infty(B,\HH^n)$ consisting of Fueter-regular functions $h=(h_1,\cdots, h_n)$ of finite $L^2$-norm, for which 
\begin{align}\label{hwt} (\partial_Y - \qj \partial_X) \sum_{i=0}^{n-1} h_i X^{n-1-i} Y^{i} \in C^\infty(B,\CC[X,Y]).\end{align}

\begin{thm}\label{thm2}For $n\geq 2$, $\scrR(B,\HH^n)$ is an irreducible unitary representation of $G$, and the quaternionic discrete series representation with minimal $K$-type $V_0\boxtimes V_n$.
\end{thm}
\begin{proof}
Under the complex-linear isomorphism
\begin{align}\label{WH} \psi: W_{n-1}=\HH\otimes V_{n-1}\rar \HH^{n},\ \ \ w=\sum_{i=0}^{n-1} q_i X^{n-1-i} Y^{i}\mapsto (q_0,\cdots, q_{n-1}),\end{align}	
the action of $u\in \HH$ on $w$ is mapped to 
	$$ u\cdot (q_0,\cdots, q_{n-1}) = (uq_0,\cdots, uq_{n-1}) \cdot {}^tR_{n-1}(u).$$
Therefore the map
$$ C^\infty(B,V_n) \rar C^{\infty}(B,\HH^n),\ \ \ \psi\circ (f_X - \qj f_Y)$$ is a $G$-equivariant injection. By Theorem (\ref{thm1}, it restricts to a $G$-equivariant embedding from the kernel of the Dirac operator $\scrD$ into the space of Fueter-regular functions in $C^\infty (B,\HH^{n})$. It's easy to check that a map $h: B \rar \HH^n$ is in this image if and only if (\ref{hwt}) is satisfied. By Schmid's thesis \cite{Sch89}, the space of $L^2$-vectors in the kernel of $\scrD$ is a discrete series representation of $G$ with minimal $K$-type $V_0\boxtimes V_n$, provided $n$ is large enough. In fact one checks that $n\geq 2$ is sufficient. That the action of $G$ is unitary, and the $L^2$-norm of $C^\infty(G,V_n)$ defines the same $L^2$-subspace as ours follows from the fact that $d\mu (1-\rN q)^{-4}$ is a $G$-invariant measure on $B$. Furthermore, the $G$-action and normalization on $C^\infty(B,\HH^n)$ is compatible with that of Liu-Zhang\cite{LiuZhang} on $C^\infty(B,V_n)$.
\end{proof}

We now give a construction of the minimal $K$-type of $\scrR(B,\HH^n)$ via regular homogeneous polynomials.

For a variable $z$, and $n$ any integer, we write 
\begin{align} z^{[n]}=\left\{\begin{array}{cl}\frac{z^n}{n!}& \text{ for }n\geq 0,\\0 & 
\text{ for } n<0.\end{array}\right.\end{align}

Recall that a function $f: \HH^\times \rar V$, for a real vector space $V$, is called \textit{homogeneous of degree $n$}, if $f(rq)=r^n f(q)$ for all $r\in \RR^\times$ and $q\in \HH^\times$. For $n\geq 0$, following \cite{Sud79}, we define
\begin{align}\label{Pkl} P_{k,l}^n: \HH \rar \CC,\ \ \ P_{k,l}^n(z+ \qj w) = \sum_{r=0}^{n} (-1)^r z^{[n-k-l+r]} \ol{z}^{[r]} 
w^{[k-r]}\ol{w}^{[l-r]},
\end{align}
and put
\begin{align} Q_{kl}^n = P_{kl}^n - \qj P_{k-1,l}^n\end{align}
These functions satisfy the Cauchy-Riemann-Fueter equations
\begin{align}\label{CRFP}\pard{P_{k,l}^n}{\ol{z}} = - \pard{P^n_{k-1,l}}{\ol{w}},\ \ \  \pard{P_{k,l}^n}{w} = \pard{P^n_{k-1,l}}{z},\end{align}
and the set 
$$ \{Q_{k,l}^n: 0\leq k\leq l\leq n\},$$
is a basis for regular homogeneous functions $\HH \rar \HH$ of degree $n$, as a right $\HH$-vector space. For $k=0,\cdots,n$, we define 
\begin{align} g_k^n: \HH \rar V_n,\ \ \ g_{k}^n(q) = \sum_{l=0}^n P^n_{l,k}(q) X^{[l]} Y^{[n-l]},\end{align}
and 
\begin{align}
h_{k}^n = (\partial_X - \qj \partial_Y) g_k^n.
\end{align}
The coordinate functions of $h_k^n$ are the $Q_{k,l}^n$, and the Cauchy-Riemann equations (\ref{CRFP}) amount to $\ol{\partial}_l h_k^n =0$. Then $h_k^n\in C^\infty(B,\HH^n)$ are regular homogeneous of degree $n$, and linearly independent. 
The action of $K\subset G$ preserves homogeneity, hence the complex span of $h_k^n$ in $\scrR(B,\HH^n)$ is a $K$-representation of dimension $n+1$. Since the minimal $K$-type has the same dimension, the two must coincide.

\section{Differential Forms Associated with Automorphic Forms}

\subsection{Fueter-Regular Automorphic Forms}

Given $\wt{f}\in C^\infty (B,W_n)$, of the form
$$ \wt{f} = \sum_{i=0}^{n} \phi_{i+1} X^{n-i} Y^i,$$
by the coordinate function of $\wt{f}$ we mean $\frakf=(\phi_1, \cdots, \phi_{n+1}): B \rar \HH^{n+1}$.

\begin{lemma}Let $f\in C^\infty(G,V_n)$, $\wt{f}=\partial_X - \qj \partial_Y \in C^\infty(B,W_{n-1})$, and let $\frakf: B \rar \HH^n$ be the coordinate function of $\wt{f}$. Then $f$ is invariant under a subgroup $\Gamma\subset G$ if and only if $\frakf$ satisfies
	$$ \frakf(\gamma q) = |cq+d|^2(cq+d)\cdot \frakf(q)\cdot R_{n-1}(cq+d),\ \ \ \text{ for all }\gamma = \gamma(a,b,c,d)\in \Gamma.$$
\end{lemma}
\begin{proof}Let $\tau = X\otimes \partial_X + Y\otimes \partial_Y \in V_1\otimes V_1^*$, and note that $\wt{f}=\Phi(\tau\otimes f)$, where $\Phi$ is the composition of $K$-equivariant maps
	$$ V_1 \otimes V_1^* \otimes V_k \overset{I\otimes P^-}\rar V_1\otimes V_{k-1} \overset{\psi\otimes I}\rar \HH \otimes V_{k-1},$$
	with $\psi: V_1 \isomto \HH$, $\psi(aX+bY)=a-\qj b$. Now $\tau$ maps to the identity $\id_{V_1}$ under the isomorphism $V_1\otimes V_1^* \rar \End(V_1)$. Since $\HH^1$ acts trivially on $\id_{V_1}\in \End(V_1)$, it does the same on $\tau\in V_1\otimes V_1^*$. Therefore, with $u\in \HH^1$, we have
	\begin{align}\label{Philem}\begin{split} \Phi(\tau\otimes \rho_{n}(u) f) &= \Phi(u\cdot (\tau\otimes f))=(\psi\otimes I) ( u \cdot (X \otimes f_X + Y \otimes f_Y)) \\
	&= (\psi\otimes I) (uX \otimes \rho_{n-1}(u) f_X + uY \otimes \rho_{n-1}(u) f_Y) \\&= u (\rho_{n-1}(u) f_X - \qj \rho_{n-1}(u) f_Y).
	\end{split}
	\end{align}
	Now with $\gamma=\gamma(a,b,c,d)\in \Gamma$, we have 
	$$|cq+d|^{2+n} \rho_{n}\left(\frac{cq+d}{|cq+d|}\right) f(q) = f(\gamma q),$$ and 
	$$ |cq+d|^{2+n}\Phi(\tau\otimes \rho_{n}\left(\frac{cq+d}{|cq+d|}\right)f(q)) = \Phi(\tau\otimes f(\gamma q)) = \wt{f}(\gamma q),$$
	therefore
	$$ \wt{f}(\gamma q) = |cq+d|^{n+2}\frac{(cq+d)}{|cq+d|}\left(\rho_{n-1}\left(\frac{cq+d}{|cq+d|}\right)f_X - \qj \rho_{n-1}\left(\frac{(cq+d)}{|cq+d|}\right) f_Y\right).$$
	Then 
	$$ \frakf(\gamma q) = |cq+d|^{n+2} \frac{(cq+d)}{|cq+d|}\frakf(q)\cdot R_{n-1}\left(\frac{cq+d}{|cq+d|}\right)=|cq+d|^2 (cq+d) \cdot \frakf(q) \cdot R_{n-1}(cq+d).$$ 
\end{proof}

Let $D$ be a definite quaternion algebra over $\QQ$, and $\calO=\calO_D \subset D$ a maximal order. For any commutative ring $R$, we put
\begin{align}\label{GD} G(R) = \left\{ g\in M_2(\calO\otimes R): g^* \mtwo{1 & \\ & -1} g = \mtwo{1 & \\ & -1}\right\}.\end{align}
We choose an embedding $D\subset \HH$ and identify $\calO\otimes \RR$ with $\HH$. Then $G(\RR)=\Sp(1,1)$ has well-defined subgroups $G(\QQ)$ and $G(\ZZ)$. For $N>0$ we define the $N$th \textit{principal congruence subgroup} of $G(\QQ)$ to be
\begin{align} \Gamma(N)=\{g\in G(\ZZ): g\equiv I\ (\mod N\calO)\} =\ker(G(\ZZ)\rar G(\ZZ/N\ZZ)).\end{align}
As usual, an \textit{arithmetic subgroup} $\Gamma\subset G(\QQ)$ is one that is commensurable with $\Gamma(N)$ for some $N$.

The above lemma suggests the following definition.

\begin{defn}Let $\Gamma \subset G(\QQ)$ be an arithmetic subgroup. A function $h: B \rar \HH^n$ is an \textit{automorphic form} on $B$ of \textit{weight $n$} and \textit{level $\Gamma$} if
	\begin{itemize}
		\item[(1)] $h(\gamma\cdot q) = |cq+d|^2(cq+d)\cdot h(q) \cdot R_{n-1}(cq+d)$, for $\gamma=\gamma(a,b,c,d)\in \Gamma$.
		\item[(2)]$|f(q)|$ has moderate growth as $|q|\rar 1$.
	\end{itemize}
It is called \textit{regular} if it is left-Fueter regular, i.e. $\ol{\partial}_l h = 0$.
\end{defn}

We have distinguished differential $1$-form
$$ dq = dt + \qi dx + \qj dy + \qk dz,$$
and a $3$-form
$$ Dq = dx \wedge dy \wedge dz -\qi dt \wedge dy \wedge dz - \qj dt \wedge dx \wedge dz - \qk dt \wedge dx \wedge dy.$$

\begin{prop}\label{formdform}For $\gamma=\gamma(a,b,c,d)\in\Sp(1,1)$, 
	\begin{itemize}
		\item[(a)] $\gamma^*(dq) = (\ol{a}+q\ol{b})^{-1} dq (cq+d)^{-1},$
		\item[(b)] $\gamma^*(d\ol{q}\wedge dq) = \ol{(cq+d)}^{-1} d\ol{q}\wedge dq (cq+d)^{-1}.$
		\item[(b)] $\gamma^*(Dq)= (\ol{a}+q\ol{b})^{-1} Dq(cq+d)^{-1}|cq+d|^{-4}.$
	\end{itemize}
\end{prop}
\begin{proof}We prove $(c)$, and leave the first two as exercise.
	For $q\in \HH^\times$, we write
	\begin{align}\label{prime}q' = \frac{q^{-1}}{\rN q},
	\end{align}
	and note that
	$$ (q')'=q,\ \ \ (pq)' = q' p', p,q\in \HH^\times,$$
	and $(rq)'=r^{-3}q'$ for $r\in \RR^\times$.
	
	For $b\in B$, $\alpha,\delta\in \HH^1$, $\gamma\in \Sp(1,1)$, write
	$$f_\gamma(q) = \gamma \cdot q,\ \ \ g_b(q) = (q+b)(\ol{b}q+1)^{-1},\ \ \iota(q)=q^{-1},\ \ \ \nu_{a,d}(q) = aqd^{-1}.$$
	In fact $g_b = f_{\gamma}$ where $\gamma = \sigma b$, and $\nu_{a,d}=f_{\gamma}$ for $\gamma = \mtwo{a & \\ & d}.$ For the relations
	$$ \iota^*(Dq) = -q' (Dq) q',\ \ \ \nu_{a,d}(q) = (a^{-1})' Dq d',$$
	we refer to \cite{Sud79}. We first claim that
	\begin{align} g_b^*(Dq) = (1+q\ol{b})' Dq (1+\ol{b}q)' (1-\rN b)^3.\end{align}
	For $b\in B$, we have
	$$ \mtwo{1 & b \\ \ol{b} & 1}  = \mtwo{& 1 \\ 1 } \mtwo{1 & \ol{b}\\ & 1} \mtwo{& 1 \\  1&} \mtwo{1 & \\ & 1-\rN b} \mtwo{1 & b \\ & 1}.$$
	Therefore, letting $\tau_b(q) = q+ b$, 
	\begin{align*}
	g_b^*(Dq) & = \tau_b^* \circ \nu_{1,1-\rN b}^* \circ \iota^* \circ \tau_{\ol{b}}^* \circ \iota^*(Dq) = \tau_b^* \circ \nu_{1,1-\rN b}^* \circ \iota^* \circ \tau_{\ol{b}}^* (-q' (Dq) q')\\
	&= \tau_b^* \circ \nu_{1,1-\rN b}^* \circ \iota^* \circ  (-(q+\ol{b})' Dq ( q+\ol{b})') = \tau_b^* \circ \nu_{1,1-\rN b}^* (-(q^{-1}+\ol{b})' (-q')(Dq)(q') (q^{-1}+\ol{b})')\\
	& = \tau_b^*\circ \nu_{1,1-\rN b}^* ((1+q\ol{b})' (Dq) (1+\ol{b}q)') = \tau_b^* \left(1+\frac{q\ol{b}}{1-\rN b}\right)' (Dq) (1-\rN b)^{-3} \left(1+\frac{\ol{b}q}{1-\rN b}\right)'\\
	&= \left(1 + \frac{q\ol{b} + \rN b}{1-\rN b}\right)' (Dq) (1-\rN b)^{-3}\left(1+\frac{\ol{b}q+ \rN b}{1-\rN b}\right)'= (1+q\ol{b})' Dq (1+\ol{b}q)' (1-\rN b)^3,
	\end{align*}
	which proves the claim. Now using the fact that $\gamma\in \Sp(1,1)$, we write
	$$ \gamma = \mtwo{a & b \\ c & d} = \mtwo{ 1 & bd^{-1}\\ \ol{d}^{-1}\ol{b}& 1} \mtwo{a & \\ & d}.$$
	Then
	\begin{align*} \gamma^*(Dq) &= \nu_{a,d}^*\circ g_{bd^{-1}}^* (Dq) = \nu_{a,d}^*\left( (1+q\ol{d}^{-1}\ol{b})(Dq) (1+\ol{d}^{-1}\ol{b})'(1-N(bd^{-1}))^3\right)\\
	&= (1+aqd^{-1}\ol{d}^{-1} \ol{b})'(a^{-1})' (Dq) (d)' (1+\ol{d}^{-1}\ol{b} aqd^{-1})' (1-\rN (bd^{-1}))^3\\
	&= \left(\frac{\ol{a} + q\ol{b}}{\rN d}\right)' (Dq) (cq+d)' \frac{1}{\rN(d)^3} = (\ol{a} + q\ol{b})' (Dq) (cq+d)',
	\end{align*}
	where in the second-to-last equality we have used the facts $1- N(bd^{-1}) = N(d)^{-1}$ and $\ol{d}^{-1}\ol{b}a = c$, which holds since $\gamma\in \Sp(1,1)$.
\end{proof}

Let $h: B \rar \HH^n$ be an automorphic form of weight $n$ and level $\Gamma$. Let $f\in C^\infty(B,V_{n})$ be such that $f_X - \qj f_Y = h$. Then we know that $f$ satisfies
$$ f(\gamma q) = |cq+d|^2\rho_{n}(cq+d) f(q),$$
therefore by Proposition \ref{dagprop}, the function $f^{\dag}(q)=f(\ol{q})$ satisfies
$$ f^{\dag}(\gamma q) = |a+b\ol{q}|^2 \rho_{n}(a+b\ol{q}) f^{\dag}(q).$$
Writing $h^{\dag} = f^{\dag}_X - f^{\dag}_Y$, we have
$$ h^{\dag}(\gamma q) = v(\gamma,q) h^{\dag}(q) \cdot  R_{n-1}(v),$$
therefore $h^*(q) = {}^t \ol{h(\ol{q})}={}^t \ol{h}^{\dag}(q)$ satisfies
$$ h^*(\gamma q) = R_{n-1}(a+b\ol{q})^* h^*(q) (\ol{a} + q\ol{b})|a+b\ol{q}|^2.$$

\begin{prop}\label{formal} Let $f, g: B \rar \HH^{n}$, be automorphic forms of level $\Gamma$ and weight $n$. To $(f,g)$ we associate the differential forms 
\begin{align*}
\eta_{f,g} =  \frac{g^* dq f}{(1-\rN q)^2},\ \ \ \ 
\theta_{f,g} = \frac{{}^t \ol{g} d\ol{q} \wedge dq f}{1-\rN q},\ \ \ \  \omega_{f,g} = g^* Dq f.
\end{align*}
For $\gamma = \gamma(a,b,c,d)\in \Gamma$, we have
$$\gamma^* \theta_{f,g} = R_{n-1}(cq+d)^* \cdot \xi \cdot R_{n-1}(cq+d),$$
and if $\xi=\eta_{f,g}$ or $\omega_{f,g}$, we have
	$$\gamma^* \xi = R_{n-1}(a+b\ol{q})^* \cdot \xi \cdot R_{n-1}(cq+d).$$ If $f$ and $g$ are Fueter-regular, $\omega_{f,g}$ is closed. 
\end{prop}

\begin{proof}The transformation rules follow from  Proposition \ref{formdform}, and the preceding discussion, plus the fact that
	$$ \frac{1}{|cq+d|^2} = \frac{1-\rN(\gamma q)}{1-\rN q}.$$ If $g$ is (left-)regular, $g^*$ is right-regular (See Lemma \ref{reglem}). The last statement only depends on the regularity of $f$ and $g$. If $f_i$, $g_j: \HH \rar \HH$ are the individual coordinates, then $g_j^* (Dq) f_i$ are the coordinates of $\omega_{f,g}$, and
$$ d(g_j^* Dq f_i) = ((\ol{\partial}_r g_j^*)f_i + g_j^*(\ol{\partial}_l f_i) )(dt \wedge dx \wedge dy \wedge dz) = 0,$$
hence $d(\omega_{f,g})=0$.
\end{proof}

Note that each form in the proposition takes values in an $n\times n$ matrix of single-valued differential forms, with entries $a_{ij}=g_j^* \xi f_i$. It's easy to see that $\xi$ is zero if and only if either $f=0$ or $g=0$.

\subsection{A Representation of $\GL_2(\HH)$}

Differential forms that transform as in Proposition \ref{formdform} should correspond to sections of vector bundles on $\Gamma\backslash B$ associated with representations of $\Gamma$. Here we construct the representation that we expect should correspond to $\theta_{f,g}$.

Let $L(X,Y)$ be a column vector with entries $(X^n, X^{n-1}Y,\cdots, Y^n)$, and recall that 
$$ R_{n}: \GL_2(\CC) \rar \GL_{n+1}(\CC),$$ 
is defined by the relation
\begin{align} L(aX+bY, cX + dY) = R_{n}(g)\cdot L(X,Y),\ \ \ g=g(a,b,c,d)\in \GL_2(\CC).\end{align}
If $P=(c_0,\cdots, c_n)$ is the column vector of coefficients of $f(X,Y)\in \CC[X,Y]_n$, then $f(X,Y)={}^tP\cdot L(X,Y)$, and $g\cdot f(X,Y) = {}^tP \cdot L(aX+bY,cX+dY) = {}^tP\cdot R_n(g)\cdot L(X,Y) = ({}^t R_n(g)\cdot P) \cdot L(X,Y)$. Therefore ${}^tR_n$ is the matrix of the symmetric $n$th power representation of $\GL_2(\CC)$ on $\CC[X,Y]_n$ , with respect to the entries of $L(X,Y)$ as basis.
For $u\in\HH^\times$, we write $R_n(u)$ for $R_n(\iota(u))$. We choose a hermitian matrix $J_n$ such that 
\begin{align}\label{RJn} {}^t R_{n}(u) J_n \ol{R_{n}(u)}=J_n,\ \ u\in \HH^1.\end{align}

We consider $\HH$ as a vector space via right-multiplication by $\CC\subset \HH$, and identify $\HH^2$ with $\HH\otimes \CC^2$. The elements of $S^n(\HH^2)$ are linear combinations of tensors of the form
$$ w = \bfw_{1}\odot \cdots \odot \bfw_{n},\ \ \ \bfw_{i}=\mtwo{x_{i}\\ y_{i}}\in \HH^2.$$
The algebras $M_2(\CC)$ and $\HH$ acts $\CC$-linearly on $S^n(\HH^1)$, on the left and right respectively, with
$$ g\cdot w \cdot u = (g\bfw_1 u)\odot \cdots \odot (g\bfw_n u),\ \ g\in M_2(\CC),\ u\in\HH^\times.$$
The isomorphism $\CC^2\otimes \HH = \HH^2$ induces an embedding
$$ S^n(\CC^2)\otimes S^n(\HH) \rar S^n(\HH^2)$$
which is equivariant with respect to action of the $\RR$-algebra $M_2(\CC)\otimes_{\RR} \HH^\op$. 

\begin{prop}The $\RR$-linear action of the algebra $M_2(\HH)$, by left-multiplication on $\HH^2$, induces an action on $S^n(\HH^2)$, which preserves the image of $S^n(\CC^2)\otimes S^n(\HH)$. 
\end{prop}

\begin{proof}
	Since
	$$ M_2(\HH) = M_2(\CC) + \qj M_2(\CC),$$
	the action of $M_2(\HH)$ on $S^n(\HH^2)$ is determined by the action of $M_2(\CC)\otimes_{\RR} \HH^\op$, along with left-multiplication by $\qj$. Since the latter commutes with the action of $\frakS_{m}$ on $T^m(\HH^2)$, it preserves each isotypic component of $S^n(\HH^2)$ as a representation of $\GL_2(\CC)\otimes \HH^\op$. Since $S^n(\CC^2)\otimes S^n(\HH)$ is an irreducible component of $S^n(\HH^2)$ with multiplicity one, it is stabilized by $\qj$, and hence also by $M_2(\HH)$.
\end{proof}

For applications, we require an explicit description of the action of $M_2(\HH)$ on $S^n(\CC^2)\otimes S^n(\HH^2)$.

Let $$\psi_1: S^n(\CC^2) \rar \CC^{n+1},\ \ \ \psi_2: S^n(\HH)\rar \CC^{n+1},$$
be fixed isomorphisms, $\psi_1$ corresponding to the basis $e_1^{\odot n-i}e_2^{\odot i}$, $i=0,\cdots, n$, and $\psi_2 = \psi_1 \circ \psi_0$, where $\psi_0: S^n(\HH) \rar S^n(\CC^2)$ is induced by $\HH \rar \CC^2,\ z+\qj w\mapsto ze_1 -\ol{w}e_2$. Here $\HH$ is a $\CC$-vector space under left-multipication by $\CC\subset \HH$. Considering $\psi_1$, $\psi_2$ as valued in column vectors, we put
$$ M_0: S^n(\CC^2)\otimes  S^n(\HH) \rar M_{n+1}(\CC),\ \ M_0(v\otimes u)=\psi_1(v)\cdot {}^t \psi_{2}(u).$$
Then for $g\in \GL_2(\CC)$ and $h\in \HH^1$, we have
$$ M_0(gv\otimes uh) = {}^tR_{n}(g) M_0(v\otimes u) R_n(h).$$
Define an $\RR$-linear action of $M_2(\HH)$ on $M_{n+1}(\CC)$ by

$$\mu_0 (g_0 + \qj g_1)\cdot M_0(v\otimes u) =M_0(g_0 v\otimes u) +  M_0(\ol{g}_1 \ol{v}\otimes \qj u).$$

It can be realized as matrix multiplication as follows.
$$ M: S^n(\HH)\otimes S^n(\CC^2) \rar M_{2n+2,n+1}(\CC),\ \ \ M(u\otimes v) = \mone{M_0(u\otimes v)\\ M_0(\ol{u}\otimes \qj v)}.$$ For $g=M_2(\HH)$, write $g = g_0 + \qj g_1 $, with $g_0, g_1\in M_2(\CC)$, and define
$$ \mu: M_2(\HH) \rar \GL_{2n+2}(\CC),\ \ \ \mu(g) = \mtwo{{}^tR_n(g_0)&  {}^tR_n(\ol{g}_1)\\ {}^tR_n(g_1) & {}^tR_n(\ol{g}_0)}.$$

Then
$$ \mu(g)M(v\otimes u)  = \mtwo{ M_0(g_0 v\otimes u) +  M_0(\ol{g}_1 \ol{v}\otimes \qj u)\\ M_0(g_1 v\otimes u) + M_0(\ol{g}_0 \ol{v}\otimes \qj u)} = \mu(g_0) M(v\otimes u) + \qj \cdot (\mu(g_1)M(g_1v\otimes u)).$$

For $x,y\in \HH$, let 
$$ w(x,y) = \sum_{i=0}^{n} e_1^{\odot n-i} e_2^{\odot i}\otimes x^{\odot n-i} y^{\odot i} \in S^n(\CC^2)\otimes S^n(\HH),$$
and put
$$ W(x,y) = M_0(w(x,y)) \in M_{n+1}(\CC).$$

For $u\in \HH^\times$, we have
$$  W( xu,yu) = W(x,y)R_n(u) ,$$
and for $g=g(a,b,c,d)\in M_2(\HH),$
$$ \mu(g) \cdot W(x,y) = W(ax+by,cx+dy).$$

Now for $q\in \HH$, set
\begin{align} Z(q) = W(q,1).\end{align}
If $x,y\in \HH$, $y\neq 0$, and $\gamma=\gamma(a,b,c,d)\in \GL_2(\HH)$, setting $q=xy^{-1}$, we have
\begin{align} Z(\gamma q)R_n(cq+d)= \mu(\gamma) Z(q)\end{align}
as well as
$$ R_n(cq+d)^* Z(\gamma q)^* = Z(q)^* \mu(\gamma)^*.$$ 

Suppose that $\omega$ is a quaternion-valued differential $k$-form on $B$, with values in $M_{n}(\HH)$, such that for each $\gamma\in \Gamma$, 
$$ \gamma^*(\omega) = R_n(cq+d)^* \cdot \omega \cdot R_n(cq+d).$$

Then we expect $\omega$ should correspond to a section of a vector bundle on $\Gamma\backslash B$ associated with the representation $\mu^* \otimes \Omega^k \otimes \mu$ of $\Gamma$, where $\Omega^k$ denotes the $\HH$-valued $k$-forms on $B$.

\section{Appendix: Quaternionic Analysis}

In this section we collect some facts about quaternionic analysis that we use throughout the article. The material occasionally repeats \cite{Sud79}, which is a basic reference, but is mostly supplementary.

\subsection{Derivatives and Differentials}

\subsubsection{Differentiation Rules}

It's now convenient to write a quaternion $q$ as $e_0 x_0 + e_1 x_1 + e_2 x_2 + e_3 x_3$, where $x_i\in \RR$, and $(e_0,e_1,e_2,e_3)=(1,\qi,\qj,\qk)$. If $U$ and $V$ are open subsets of Euclidean space, as usual $C^1(U,V)$ denotes continuously differentiable functions $U \rar V$. We write $f\in C^1(U,\HH)$ as 
\begin{align} f = \sum_{i=0}^3 e_i f^{(i)},\ \ \ f^{(i)}\in C^1(U,\RR)\end{align}
where $f^{(i)}$ are the coordinates of $f$ with respect to $e_0,\cdots, e_3$.

If $f$ depends on $x\in U\subset \RR$, we write $f_x = \frac{df}{dx}$. 

\begin{prop}Let $U\subset \RR$ be an open interval. For $f,g\in C^1(U,\HH)$, the differentiation rules are
	\begin{enumerate}
		\item[(a)] (Product) $(fg)_x = f_x g + fg_x$
		\item[(b)] (Inverse) $(f^{-1})_x = -f^{-1}f_x f^{-1}$ 
		\item[(c)] (Quotient) $(fg^{-1})_x = f_x g^{-1} - f g^{-1} g_x g^{-1},\ \ \ (g^{-1}f) = -g^{-1}g_xg^{-1}f + g^{-1}f_x$
	\end{enumerate}	
\end{prop}
\begin{proof} We only mention that 
	$(a)$ follows from the equality of both sides with
	$$ \sum_{i,j} e_i e_j (f_{x}^{(i)}g^{(j)} + f^{(i)} g_{x}^{(j)}),$$
and omit the rest.
\end{proof}

For a quaternion $q=\sum_{i} x_i e_i\in \HH$, let us for the moment write
\begin{align} [q] = \mone{x_0 \\ \vdots \\ x_3}.\end{align}

We may identify $\HH$ with $\RR^4$ via 
$\RR^4 \rar \HH,\ \ \ (x_0,\cdots, x_3)\mapsto \sum_{i} x_i e_i.$

If $U\subset \HH$ is open and $f\in C^1(U,\HH)$, considering $U$ as a subset of $\RR^4$, we may write the Jacobian of $f$ as a row-vector with quaternion-valued functions for entries
\begin{align}\label{Jacobian}Df = \left(\begin{array}{ccc}f_{x_0}& \cdots& f_{x_3}\end{array}\right).\end{align}
The usual Jacobian as a function $U \rar \RR^4$ is obtained by replacing each $f_{x_i}$ with $[f_{x_i}]$. 

For $f,g\in C^1(U,\HH)$, we have
$$ \frac{d}{dx} (f\circ g) = \sum_i e_i \frac{d}{dx}(f^{(i)} \circ g) = \sum_{i} e_i \sum_j (f^{(i)}_{x_j}\circ g)g^{(j)}_x = \sum_j (f_{x_j}\circ g) g^{(j)}_x=((Df)\circ g)\cdot [g_x],$$
where $\cdot$ on the right-hand side is matrix multiplication.

\subsubsection{Differential operators}

Now we write $q=t+\qi x + \qj y + \qk z\in \HH$. Let $U\subset \HH$ be open, and let $V$ be a finite-dimensional two-sided $\HH$-vector space. We have $\HH$-linear differential operators acting on $f\in C^1(U,V)$ on the left,
\begin{align} 
\begin{split}\partial_l f = \pard{f}{t}- \qi \pard{f}{x} -\qj \pard{f}{y} -\qk \pard{f}{z},\\ \ol{\partial}_l f =\pard{f}{t}+ \qi \pard{f}{x} +\qj \pard{f}{y} +\qk \pard{f}{z},
\end{split}
\end{align}
and their right-sided counterparts 
\begin{align}
\begin{split}
\partial_r f = \pard{f}{t} - \pard{f}{x}\qi - \pard{f}{y}\qj - \pard{f}{z}\qk,\\ \ol{\partial}_r f = \pard{f}{t} + \pard{f}{x}\qi + \pard{f}{y}\qj + \pard{f}{z}\qk.
\end{split}
\end{align}
The first two operators are right $\HH$-linear, and the last two left $\HH$-linear. We say a function $f\in C^1(U,V)$ is \textit{left-regular} if $ \ol{\partial}_l f = 0,$ and \textit{left-anti-regular} if $\partial_l f = 0$. We say it is \textit{right-regular} if $\ol{\partial}_r f = 0$, and \textit{right-anti-regular} if $\partial_r f = 0$. By a \textit{regular} or \textit{Fueter-regular} function we always mean \textit{left-regular}. Likewise an \textit{anti-regular} function is always \textit{left-anti-regular}. 

Fix a real form $V_0\subset V$ for $V$ so that $V=\HH\otimes_\RR V_0 \otimes_\RR \HH$, and define conjugation $v \mapsto \ol{v}$ on $V$ by $a\otimes v_0 \otimes b \mapsto \ol{b} \otimes v_0 \otimes \ol{a}.$ We then have involutions
$$ C^1(U,V) \rar C^1(U,V),\ \  f \mapsto \ol{f},\ \ \ \ol{f}(q) = \ol{f(q)},$$
and
$$ C^1(U,V) \rar C^1(\ol{U},V),\ \ f\mapsto f^\dag,\ \ \ f^\dag(q)=f(\ol{q}).$$

Note that $f\mapsto f^\dag$ is left and right $\HH$-linear, but $f\mapsto \ol{f}$ is only $\RR$-linear. It's easy to verify that
\begin{align}\label{oprels}
& &\ol{\partial_l f} &= \ol{\partial}_r \ol{f}, & \ol{\partial_r f} &= \ol{\partial}_l \ol{f},& &\\\label{oprels2}
& &\ol{\partial}_l f^\dag &= (\partial_l f)^\dag, & \ol{\partial}_r f^\dag 
&= {(\partial_r f)}^\dag.& &
\end{align}

We also define 
\begin{align}\label{fdag} C^1(U,V) \rar C^1(U,V),\ \ \ f \mapsto f^*,\ \ \ f^*(q) = \ol{f(\ol{q})}\end{align}
so that $f^* = \ol{f}^\dag$. 

\begin{lemma}\label{reglem}For $f\in C^1(U,V)$, we have
	$$ f\text{ is left-regular} \iff \ol{f} \text{ is right anti-regular} \iff f^\dag \text{ is left anti-regular} \iff f^* \text{ is right-regular}.$$
\end{lemma}
\begin{proof}Follows easily from (\ref{oprels}) and (\ref{oprels2}), and the definition (\ref{fdag}).
\end{proof}
By applying the lemma to $f^{\dag}$, $f^*$, and $\ol{f}$ in place of $f$, we obtain similar relations that contain the same information. As a mnemonic it may be helpful to have in mind the following diagram:
\begin{center}
	\begin{tabular}{|cc|cc|}
		\hline &&& \\
		&left regular & right regular &\\ & &&\\
		\hline & & & \\
		&left anti-regular & right anti-regular&\\ &&& \\
		\hline
	\end{tabular}
\end{center}
and note that
\begin{itemize}
	\item[(1)] $f \mapsto f^*$ reflects across the vertical axis, switching ``left'' and ``right'',
	\item[(2)] $f\mapsto f^\dag$ reflects across the horizontal axis, adding or removing ``anti'',
	\item[(3)] $f\mapsto \ol{f}$ rotates by $\pi$, switching ``left'' and ``right'', and also adding or removing ``anti''.
\end{itemize}

\subsubsection{Differential forms}

For $k=0,\cdots, 4$, let $\Omega_\RR^k$ denote the space of real-valued differential forms on $\HH\simeq \RR^4$. They are generated by the forms $dx_i$, $i=0,\cdots, 3$ under the exterior product. By $\Omega^k$ we denote the two-sided $\HH$-vector space containing $\Omega_\RR^k$ such that $\dim_{\HH} \Omega^k =\dim_\RR \Omega_\RR^k$. In particular, if $\omega_0\in \Omega_{\RR}^k$ is real-valued, and $p,q\in \HH$, we have
$$ p\omega_0q = pq \omega_0 = \omega_0 pq$$
in $\Omega^k$.

The exterior product of $\eta\in \Omega^r$ and $\omega\in \Omega^s$ is defined by
$$ (\eta \wedge \omega)(v_1,\cdots, v_{r+s}) = \frac{1}{r! s!} \sgn(\sigma)\sum_{\sigma\in S_{r+s}} \eta(v_{\sigma(1)},\cdots v_{\sigma(r)}) \omega(v_{\sigma(r+1)},\cdots v_{\sigma(r+s)}).$$
It satisfies
$$ (p \omega) \wedge ( \eta q) = p (\omega \wedge \eta) q,\ \ \ (\omega p)\wedge \eta = \omega\wedge (p \eta),\ \ \ p,q\in\HH.$$
In general, $\omega \wedge \eta \neq \pm \eta \wedge \omega$, unless the values of $\omega$ and $\eta$ commute, which is the case for instance if either one is real-valued. For $\omega\in \Omega^r$ and $h\in \HH$, we have
\begin{align}\label{dxcom} \omega \wedge h dx_i = (\omega h) \wedge dx_i = (-1)^r dx_i \wedge \omega h,\ \ i=0,\cdots, 3.
\end{align}

The $1$-form that corresponds to $\id\in \Hom_{\RR}(\HH,\HH)=\Alt^1(\HH,\HH)$ is  is
\begin{align}\label{dq} dq = \sum_{i} e_i dx_i = t + \qi dx + \qj dy + \qk dz.
\end{align}

We have
\begin{align} dq \wedge dq = \qi dy \wedge dz + \qj dz \wedge dx+ \qk dx\wedge dy.\end{align}
For $a,b,x\in\HH$, 
$$ (dq\wedge dq)_{x}(a,b) = ab - ba.$$

Let $f: \HH \rar \HH$ be a smooth function, and $Df:T(\HH) \rar T(\HH)$ the derivative. For each $\omega\in\Omega^r$ the pullback $f^*(\omega)$ of $\omega$ by $f$ is defined by
\begin{align}  (f^*(\omega))_q = \omega_q \circ (Df)_q.\end{align}
The differential of $f$ is 
$$ df = f^*(dq) = \sum_{i=0}^3 \frac{\partial f}{\partial x_i} dx_i.$$

Then, using (\ref{dxcom}), we have
\begin{align}\begin{split} f^*(dq \wedge dq) &= \sum_{i,j} \pard{f}{x_i} dx_i \wedge \pard{f}{x_j} dx_j = \sum_{i,j=0}^3 \pard{f}{x_i} (dx_i \wedge dx_j) \pard{f}x_j\\
&= \sum_{i<j} (\pard{f}{x_i} \pard{f}{x_j} - \pard{f}{x_j}\pard{f}{x_i}) dx_i\wedge dx_j.\end{split}\end{align}

We have
$$d\ol{q}\wedge dq = \sum_{i<j} (\ol{e}_i e_j -\ol{e}_j e_i) dx_i \wedge dx_j.$$
If $i=0$, we have $\ol{e}_i e_j - \ol{e}_j e_i = 2e_j$. If $i,j\neq 0$, $\ol{e}_i e_j - \ol{e}_j e_i = -e_i e_j + e_j e_i = - 2e_i e_j$. Therefore
$$ \frac{1}{2} d\ol{q}\wedge dq = (\qi dt \wedge dx + \qj dt \wedge du + \qk dt \wedge dz) - (\qk dx \wedge dy + \qi dy \wedge dz -\qj dx\wedge dz).$$
It follows that 
\begin{align} d\ol{q} \wedge dq \wedge df &= (-\qk dx \wedge dy \wedge dt + \qi dy \wedge dz \wedge dt -\qj dx \wedge dz \wedge dt)\pard{f}{t}
\end{align}

There's a distinguished $3$-form on $B$ given by
$$ Dq = dx \wedge dy \wedge dz -\qi dt \wedge dy \wedge dz - \qj dt \wedge dx \wedge dz - \qk dt \wedge dx \wedge dy,$$
and the standard $4$-form
$$ \omega_0 = dt \wedge dx \wedge dy \wedge dz.$$

Suppose $V$ is an $\HH$-vector space, and $f: \HH \rar V$ a smooth map. Then $f$ is regular if and only if 
$$ Dq \wedge f = 0.$$

For functions $f,g: \HH \rar \HH$, we have an identity
\begin{align} d (g  Dq  f) = dg \wedge Dq f - g Dq \wedge df = ((\ol{\partial}_r g) f - g(\ol {\partial}_l f))\omega_0.\end{align}

Setting $g=1$, so that $dg=0$, we obtain
$$ d(Dq f) = - Dq \wedge df.$$
Then $f$ is regular in a domain $U\subset \HH$ if and only if the form $Dq f$ is \textit{closed} on $U$. More generally, if $g$ is right-regular, and $f$ left-regular, then $gDq f$ is closed.

\subsection{Integrals and Series Expansions}

Now suppose $U$ is star-shaped, so that $Dq f$ is closed on $U$ if and only if it's exact. Let $C$ be a $3$-chain contained in $U$ whose homology class $[C]\in H^3(U,\ZZ)$ vanishes. Then $f$ is regular in $U$ if and only if $Dq f=d\eta$ for some $2$-form $\eta$ on $U$, in which case
$$ \int_{C} Dq f = \int_{C} d\eta = \int_{\partial C} \eta$$

For $q\in \HH$, $r>0$, denote the open disk of radius $r$ around $q$ by
$$ B_r(q) = \{ p\in \HH: |p-q|<r\},$$
and the punctured ball of radius $r$
$$B_r(q)^\times = B_r(q)-\{q\}.$$
For $0<r_1<r_2$ let
$$ U_{r_1,r_2}(q) = \{ p\in \HH: r_1\leq |p-q|\leq r_2\},$$
denote the closed annulus of radius $(r_1,r_2)$ centered at $q$. 

Let $\nu=(n_1,n_2,n_3)\in \ZZ^3_{\geq 0}$. There exist regular homogeneous functions $P_{\nu}$ on $\HH$, of homogeneous degree $n=n_1+n_2+n_3$, and $G_{\nu}$ on $\HH^\times$ of homogeneous degree $-n-3$, with the following property \cite[Theorem 11]{Sud79}.

\begin{thm}[Laurent Series]Let $f: B_r(q_0)^\times \rar \HH$ be regular. Then:
	\begin{itemize}\item[(a)]There exist 
		\textit{unique} $a_{\nu}, b_{\nu}\in\HH$, indexed by $\nu\in \ZZ_{\geq 0}^3$, such that
		$$ f(q) = \sum_{n=0}^\infty \sum_{|\nu|=n} P_{\nu}(q-q_0)a_\nu + G_{\nu}(q-q_0) 
		b_\nu,$$
		the series converging absolutely on $B_r(q)^\times$, and uniformly on $U_{r_1,r_2}(q_0)$ with 
		$0<r_1<r_2<r$. 
		\item[(b)]The coefficients $a_{\nu}$, $b_{\nu}$ are determined by the 
		integral 
		formulas
		$$ a_{\nu} = \frac{1}{2\pi^2} \int_{C} G_{\nu}(q-q_0) Dq f(q),\ \ \ 
		b_{\nu}= \frac{1}{2\pi^2} \int_{C} P_{\nu}(q-q_0) Dq f(q),$$
		where $C\subset B_r(q_0)^\times$ is any closed $3$-chain homologous to 	$\partial B_{r_0}(q_0)$ with $r_0<r$. 
		\item[(c)]The function $f$ can be 
		extended to a regular function $B_r(q_0)\rar \HH$ if and only if $b_{\nu}=0$ 
		for all $\nu$.
	\end{itemize}
\end{thm}
\begin{cor}[Liouville's Theorem] A bounded regular function $f: \HH \rar \HH$ 
	is constant.
\end{cor}
\begin{proof}This follows from the same argument as in the complex case.
\end{proof}

\bibliographystyle{alpha}
{\small \bibliography{../../headers/refdb}}

\end{document}